\documentclass[a4paper,11pt]{amsart}
\usepackage[utf8]{inputenc}
\usepackage[british]{babel}
\usepackage{dsfont}
\usepackage{xcolor}
\usepackage{amsmath,amssymb,amstext,amsfonts,color,mathtools,soul}
\usepackage[colorlinks=true,citecolor=black,linkcolor=black]{hyperref}
\usepackage{amsthm}
\usepackage{backref}
\usepackage{bigints}

\usepackage{hyperref}

\newtheorem{thm}{Theorem}

\newtheorem{cor}{Corollary}

\theoremstyle{definition}

\theoremstyle{remark}
\newtheorem{rem}{Remark}

\begin{document}

\title[Poisson-Type Local Limit Theorems]{Local Limit Theorems for Poisson's Binomial in the Case of Infinite Expectation}

\date{}

\author{Italo Simonelli\(^\dagger\), Lucia D. Simonelli\(^\star\)}
\thanks{\(\dagger\) McDaniel College,   Department of Mathematics and Computer Science, Westminster, MD, USA,
isimonelli@mcdaniel.edu}
\thanks{\(\star\) Abdus Salam International Centre for Theoretical Physics, Trieste, Italy, 
lucia.simonelli@ictp.it}
  
 \maketitle

\bigskip

\begin{abstract}
 Let $ V_{n} = X_{1,n} + X_{2,n} + \cdots + X_{n,n}$ where $X_{i,n}$ are Bernoulli random variables which take the value $1$ with probability $b(i;n)$. Let $\lambda_{n} = \sum\limits_{i=1}^{n} b(i;n) $, $\lambda = \lim\limits_{n \to \infty} \lambda_n,$ and $m_n = \max\limits_{1 \leq i \leq n} b(i;n)$. We derive asymptotic results for  $P(V_{n}=k)$ that hold without assuming that  $\lambda < +\infty$ or $m_n \to 0$.   Also,  we do not assume $k$ to be fixed, but instead, our results hold uniformly for all $k$ which satisfy particular growth conditions with respect to $n$. These results extend known Poisson local limit theorems to the case when $\lambda = +\infty$. While our results apply to triangular arrays, without the assumption that \(m_n \to 0\) they continue to hold  for sums of Bernoulli random variables.   In this setting, our growth conditions cover a range of values for $k$ not centered at $\lambda_n$, thus complementing  known local limit theorems based on approximation by the normal distribution. In addition, we show that our local limit theorems apply to a scheme of dependent random variables introduced in the work of Sevast'yanov.

\end{abstract}



\section{Introduction}

\subsection{Preliminaries}

 Let $X_{i,n}$ be Bernoulli random variables which take the value $1$ with probability $b(i,n)$; the $X_{i,n}$'s are not assumed to be identically distributed. Let

\begin{equation} \label{v}
V_{n} = X_{1,n} + X_{2,n} + \cdots + X_{n,n}
\end{equation}

 and
\[
b(i_{1},i_{2},...,i_{k};n) = P(X_{i_{1},n} = X_{i_{2},n}= \cdots = X_{i_{k},n}=1).
\]

 We denote

\begin{equation} \nonumber 
 \lambda_{n}=\sum_{i=1}^{n} b(i;n), 
 \qquad \lambda = \lim_{n \to \infty}\lambda_n,
\end{equation}

\[
\alpha_n = \sum_{i=1}^n \ln \left(1-b(i;n)\right),
 \qquad \beta_n = \sum_{i=1}^n \frac{b(i;n)}{1-b(i;n)}, \qquad m_n = \max_{1 \leq i \leq n} b(i;n).
\]

In this paper we derive Poisson-type local limit theorems for $P(V_n = k)$ that hold  when the $X_{i,n}$ are assumed to be independent Bernoulli random variables and  $\lambda = +\infty$. In particular we give sufficient conditions under which 
  \vspace{.3cm}
 
  $$P(V_n = k)\sim P(V_n =0)\, \frac{\lambda_n^k}{k!}$$
 or 
  $$P(V_n = k)\sim P(V_n =0)\, \frac{\beta_n^k}{k!}.$$
  Under additional assumptions  we also prove exact Poisson asymptotic behavior
  \begin{equation*}
 P\left(V_{n}=k\right) \sim e^{-\lambda_{n}}\, \frac{\lambda_{n}^{k}}{k!}.
\end{equation*}
While our results apply to triangular arrays, when no assumption on  $m_n \to 0$ is made, they continue to hold   for sums of Bernoulli random variables.  Also,  we do not assume $k$ to be fixed, but instead, our results hold uniformly for all $k$ which satisfy particular growth conditions with respect to $n$.  These growth conditions are valid on intervals that are not contained in those for which (\ref{intro-2}) or (\ref{intro-1}) hold, thus  complementing known local limit theorems based on approximation by the normal distribution and providing new results in the setting of sums of Bernoulli random variables. 

In addition, we show that our  local limit theorems apply to a scheme of dependent random variables introduced in the work of  \cite{Seb}, in which local limit theorems were derived under assumptions $\lambda < +\infty$ and $m_n \to 0$.

\vspace{.2cm}

In \ref{sec:setting} we provide the relevant background and motivation for this work. Section \ref{sec:theorems} contains the theorems and proofs.


\subsection{The setting and background}
\label{sec:setting}

The problem of approximating the limiting local behavior of the distribution of $V_n$, known as Poisson-binomial or Poisson's Binomial distribution, has a long history in probability. 
This problem has been studied in two main directions: 

\begin{eqnarray*}
\quad & (a)\; & \lambda = +\infty \qquad   \mbox{with  a Gaussian limit}\\[.5cm]
\quad &  (b) \;&   0 < \lambda \leq +\infty \qquad   \mbox{with a (asymptotic) Poisson limit.}
\end{eqnarray*}

 The first local limit result for  case (a) is the well known  De Moivre-Laplace Theorem (1795), see \cite{Gorroochurn} for a historical account. This theorem states that if the $X_{i,n}$ are independent and identically distributed Bernoulli random variables, with $0 < p < 1$, then as $n\rightarrow +\infty$,
\begin{equation} \label{intro-2}
\sup_{k: |k-np| \leq\varphi(n)} \left|\frac{P\left(V_n=k\right)}{\frac{1}{\sqrt{2\, \pi n p (1-p)}}e^{-(k-np)^2/(2np(1-p))}}-1\right| \longrightarrow 0
\end{equation}
for $\varphi(n) = o(np(1-p))^{2/3}$.  For a proof see \cite{Shiryaev}, for example.

Several extensions of  (\ref{intro-2}) without the assumption that the $X_{i,n}$ are identically distributed have been proposed. For example, let  $b(i;n) = p_i, 1-p_i = q_i$, and $B_n^2 =Var (V_n)$.  Under the assumption that 
$0 < p_i < 1$,  \cite{MMM}  proved the existence of an absolute constant $C$ such that

\begin{equation} \label{intro-1}
\left|B_n P\left(V_n =k\right) - (\sqrt{2\pi})^{-1}\, e^{-\frac{(k-\lambda_n)^2}{2 B_n^2}}\right| < C\, \frac{\sum_{i=1}^n p_i q_i\left(p_i^2 + q_i^2\right)}{B_n^3}.
\end{equation}
Further extensions which cover arbitrary lattice distributions can be found in  \cite{Gnedenko}, \cite{rita} and \cite{Petrov}, for example.  In such results, however, as in (\ref{intro-1}), the differences of the probabilities are considered instead of the quotients. 

It is important to note that the asymptotic  ratio of probabilities cannot be recovered from their asymptotic differences when each term in the difference goes to zero.  In fact, if one is interested  in ratios, and the random variables are i.i.d. Bernoulli, (\ref{intro-1}) is not an improvement over (\ref{intro-2}). By using the notation in (\ref{intro-2}), expression (\ref{intro-1}) reduces to 
\begin{equation} \label{intro-3}
e^{-\frac{(k-np)^2}{2np(1-p)}} \left|\frac{P\left(V_n=k\right)}{\frac{1}{\sqrt{2\, \pi n p (1-p)}}e^{-(k-np)^2/(2np(1-p))}}-1\right| < \frac{C'}{\sqrt{n p (1-p)}},
\end{equation}
where $C' \geq  C/4.$  It follows from (\ref{intro-3}) that if 
$$|k-np| > \sqrt{np(1-p)}\, \sqrt{\ln (np(1-p))}$$
the factor in front of the absolute value in (\ref{intro-3}) is eventually less than the upper bound, thus giving no useful information on the size of the ratio inside the absolute value. Hence (\ref{intro-1}) does not imply (\ref{intro-2}). 

Note that in general the  estimate in (\ref{intro-3}) cannot be improved (see \cite{Petrov}, p. 197).

Case (b), concerning Poisson limits,  was first investigated under the assumption that the $X_{i,n}$ are independent and $\lambda < +\infty$. In this setting it is well known that $V_n$ converges pointwise to a Poisson random variable with parameter  $\lambda $ if and only if 
$$(i^a)  \,    \lim\limits_{n \to +\infty} m_n = 0 \quad \mbox{ or } \quad    (i^b) \, \lim\limits_{n \to +\infty} \sum_{i=1}^n b^{2}(i;n) = 0,  $$
see \cite{wang}. 

 In general, condition $(i^b)$ implies condition $(i^a)$; in the case that $\lambda < +\infty$, conditions $(i^a)$ and $(i^b)$ are actually equivalent.

 Pointwise convergence of $V_n$ to a Poisson random variable $T$ as well as local limit theorems can be obtained by showing that the variation distance goes to $0$ as $n$ goes to $\infty$.  The variation distance, $D$, is given by

\begin{eqnarray}   \nonumber 
D = \sup_{k\geq 0} \left | P(V_{n} \leq k) - P(T \leq k) \right |.
\end{eqnarray}



 In this setting, one of the first results for pointwise convergence can be found in \cite{von}. Subsequently, various methods were employed to improve rates of convergence, or equivalently, to find tighter upper bounds for $D$, see for example \cite{Pro}, \cite{lecam}, \cite{kerstan}, \cite{verv}, \cite{simjohn}, and \cite{wang}.
 
 Poisson approximation to $V_n$ in the case when the $X_{i,n}$ are dependent has also been the object of much investigation and several approaches have been proposed; see for example \cite{hodgescam}, \cite{free}, \cite{chen}, \cite{barbour}, and \cite{arratia}. 

The case of  approximating $ P(V_n =k)$ by a Poisson random variable $T_n$ with parameter $\lambda_n$, when $\lim_{n\rightarrow +\infty} \lambda_n = +\infty$, has been considered by a smaller number of researchers; it has generated very interesting  results. 
 
  \cite{dehpfeif}, show that under condition $(i^a)$, 

\begin{equation} \label{Intro-4}
D \sim \frac{1}{\sqrt{2 \pi e}  }\frac{\sum\limits_{i=1}^{n}b^{2}(i;n)}{\lambda_n}.
\end{equation}

Much broader work has been done  by Hwang, and Kowalski and Nikeghbali. Hwang  derived a unified scheme for Poisson approximation to discrete random variables. These results are based on the assumption that the generating function of the random variables satisfies certain asymptotic regularity conditions, see \cite{Hwang}. However these results are approximation theorems, not limit theorems. 

Kowalski and Nikeghbali opened a new direction in the theory of convergence of random variables by defining a new type of convergence, called mod Poisson convergence \cite{Kowalski}. Limited to our setting however, the general asymptotic estimates their theory provides are comparable to (\ref{Intro-4}), which again do not translate into meaningful bounds for the ratio of probabilities for the range of values we consider in this work.

\section{Statement of Theorems and Proofs}
\label{sec:theorems}

\subsection{Independent Case.} We begin by considering the case of independent random variables and obtain a local limit theorem that holds for $\lambda \leq +\infty$.

\begin{thm} \label{Independent_1}
Let $V_{n}$ be random variable as defined in (\ref{v}) where the $X_{i,n}$ are independent, and  ${\displaystyle \lim_{n\rightarrow +\infty} \lambda_n = \lambda}$, $0 < \lambda \leq +\infty$.  

If 
\begin{enumerate}
\item[(A1)]
\begin{equation*}
\lim_{n\rightarrow +\infty}m_n = 0,
\end{equation*}
\end{enumerate}
then for any function $\phi(n)$ such that $\phi(n)\, m_n \to 0$ as $n \to +\infty$, 
 $$ \sup_{\{k:\,k^2\leq \> \phi(n)\}} \left|\frac{P\left(V_n = k\right)}{P(V_n=0) \lambda_n^{k}/k!} - 1 \right|\longrightarrow 0.
 $$
\end{thm}
\begin{proof}
Suppose $k \geq 1$, and let $\phi(n)$ be such that $\phi(n)\, m_n \rightarrow 0$ as $n\rightarrow +\infty$. Then
\begin{equation}
\label{ind_2}
P\left(V_n = k\right)  =  \sum_{\substack{B\subset \{1,2,\cdots , n\}\\ |B| = k}} \prod_{i\in B} b(i;n)\, \prod_{j\in B^c} (1-b(j;n)) 
\end{equation}
\[
 \hspace{1cm} \geq  P\left(V_n =0\right)\sum_{\substack{B\subset \{1,2,\cdots , n\}\\ |B| = k}} \prod_{i\in B} b(i;n).
\]
We rewrite the sum in the above lower bound as follows,
\begin{eqnarray} 
\sum_{\substack{B\subset \{1,2,\cdots , n\}\\ |B| = k}}  \prod_{i\in B} b(i;n)
& = & 
   \sum_{\substack{A\subset \{1,2,\cdots , n\}\\ |A| = k-1}} \prod_{i\in A} b(i;n) \frac{\sum_{j\in A^c} b(j;n)}{k}
\nonumber\\
& = &    \sum_{\substack{A\subset \{1,2,\cdots , n\}\\ |A| = k-1}} \prod_{i\in A} b(i;n) \frac{\lambda_n -\sum_{j\in A} b(j;n)}{k}  
\label{ind_3} \\[.3cm]
& \geq &   \frac{\left(\lambda_n - (k-1) m_n\right)}{k}\,  \sum_{\substack{A\subset \{1,2,\cdots , n\}\\ |A| = k-1}} \prod_{i\in A} b(i;n). \nonumber
\end{eqnarray}

By repeating similar calculations as above,  one obtains
\begin{equation} \nonumber
\sum_{\substack{B\subset \{1,2,\cdots , n\}\\ |B| = k}} \prod_{i\in B} b(i;n) \geq  \frac{\prod\limits_{j=0}^{k-1} \left(\lambda_n - j m_n\right)}{k!},
\end{equation}
and
\begin{eqnarray} 
P\left(V_n = k\right) &  \geq &  P\left(V_n =0\right) \> \frac{\prod\limits_{j=0}^{k-1} \left(\lambda_n - j m_n\right)}{k!} \geq   P\left(V_n =0\right)\, \frac{(\lambda_n - km_n)^k}{k!} \nonumber \\[.2cm]
& = &  P\left(V_n =0\right)\, \frac{\lambda_n^k}{k!}\,\left(1 - \frac{km_n}{\lambda_n}\right)^k 
 = P\left(V_n =0\right)\, \frac{\lambda_n^k}{k!}\,\left(1 - \frac{\frac{k^2m_n}{\lambda_n}}{k}\right)^k
\nonumber
\end{eqnarray}


\begin{equation}
\label{lb2}
\hspace{-3.3cm} \geq P\left(V_n =0\right)\, \frac{\lambda_n^k}{k!}\,\left(1 - \frac{k^2m_n}{\lambda_n}\right)
\end{equation}

for sufficiently large $n$, since ${\displaystyle 
\left(1 - \frac{x}{k}\right)^k
}$ is increasing in $k$ if $0 < x < 1$.
\vspace{.2cm}

 In order to derive an upper bound for $P(V_n =k)$,
 
\[
\begin{split}
P\left(V_n = k\right)  & =  \sum_{\substack{B\subset \{1,2,\cdots , n\}\\ |B| = k}} \prod_{i\in B} b(i;n)\, \prod_{j\in B^c} (1-b(j;n))\\
 & =   \sum_{\substack{B\subset \{1,2,\cdots , n\}\\ |B| = k}} \prod_{i\in B} \frac{b(i;n)}{1-b(i;n)}\, \prod_{j=1}^{n} (1-b(j;n))\\
 & \leq  P(V_n=0) \sum_{\substack{B\subset \{1,2,\cdots , n\}\\ |B| = k}} \prod_{i\in B} \frac{b(i;n)}{1-m_n}\\
 & =  P(V_n=0) \cdot \frac{1}{(1-m_n)^k} \sum_{\substack{B\subset \{1,2,\cdots , n\}\\ |B| = k}} \prod_{i\in B} b(i;n).
\end{split}
\]

 From (\ref{ind_3}),

\[
\sum_{\substack{B\subset \{1,2,\cdots , n\}\\ |B| = k}} \prod_{i\in B} b(i;n) \leq  \frac{\lambda_n}{k}\,  \sum_{\substack{A\subset \{1,2,\cdots , n\}\\ |A| = k-1}} \prod_{i\in A} b(i;n),
\]
and by repeating the same calculations,
\[
\sum_{\substack{A\subset \{1,2,\cdots , n\}\\ |A| = k}} \prod_{i\in A} b(i;n) \leq   \frac{\lambda_n^k}{k!}.
\]

 Since for $0 \leq x  < 1$
$$\frac{1}{(1-\frac{x}{k})^k} \leq e^{x} \leq 1 + \frac{x}{1-x},$$
the above calculation gives, for $x=k\,m_n$,
\begin{eqnarray*}
P\left(V_n = k\right)  \leq P(V_n=0) \cdot \frac{\lambda_n^k}{k!} \left(1 + \frac{k m_n}{1-km_n}\right)
\end{eqnarray*}

 if $n$ is sufficiently large. We now have

\begin{eqnarray}\label{bounds}
-\epsilon_1(k,m_n) \leq \frac{P\left(V_n = k\right)}{P(V_n=0) \lambda_n^{k}/k!} -1 \leq \epsilon_2(k,m_n),
\end{eqnarray}

 where ${\displaystyle \epsilon_1(k,m_n) = \frac{k^2 m_n}{\lambda_n}}$ and ${\displaystyle \epsilon_2(k,m_n) = \frac{km_n}{1-km_n}}$.  Thus, for $i=1,2$, as $n \to + \infty$,
$$\sup_{\{k: \> k^2 \leq \phi(n)\}} \epsilon_i (k,m_n)\rightarrow 0.$$

 This, and the validity of the result in the case $k=0$ completes the proof.

\end{proof}

\begin{rem} Theorem \ref{Independent_1} holds without any assumption on the finiteness of $\lambda$. Estimates for $m_n$ and simple additional information about $\lambda_n$, e.g. $\lambda_n \geq c > 0$ for all $n$, are sufficient to obtain accurate asymptotics for $V_n$ via (\ref{bounds}).
\end{rem}

 If  $\lambda < +\infty$, then \[\lim_{n\to +\infty} P(V_n=0) = \lim\limits_{n \to +\infty} e^{\alpha_n} = e^{- \lambda},\] and thus, Theorem \ref{Independent_1} implies a well known convergence result  which we state as a corollary.

\begin{cor} \label{corollary} Let $V_{n}$ be a random variable as defined in (\ref{v}) where the $X_{i,n}$ are independent.   If the following hold,
\begin{enumerate}
\item[(A1)]
${\displaystyle
\lim_{n\rightarrow +\infty}m_n= 0,}$

\item[(A2)]  ${\displaystyle \lim_{n\rightarrow +\infty} \lambda_n =\lambda < +\infty}$,

\end{enumerate}

then for any function $\phi(n)$ such that $\phi(n) m_n \to 0$  as $n \to + \infty,$

 \[
 \sup_{\{k:\,k^2\leq \> \phi(n)\}} \left|\frac{P\left(V_n = k\right)}{e^{-\lambda} \lambda^{k}/k!} - 1 \right|\longrightarrow 0.
\]

\end{cor}
\vspace{.2cm}

In the case one assumes $\lambda = +\infty$, we obtain a local limit theorem replacing the condition $m_n \to 0$ by the weaker assumption that $m_n$ is bounded away from $1$.

\begin{thm} \label{Independent_2}
Let $V_{n}$ be a random variable as defined in (\ref{v}) where the $X_{i,n}$ are independent, and $m_n < \beta < 1$ for all $n$.

 If 

\begin{enumerate}
\item[(A3)]
\begin{equation*}
\lim_{n\rightarrow +\infty}\lambda_n =  +\infty,
\end{equation*}
\end{enumerate}
then for any function $\phi(n)$ such that $\frac{\phi(n)}{\lambda_n} \to 0$  as $n \to +\infty$,  n$\rightarrow +\infty $, 
 $$ \sup_{\{k:\,k^2 \leq \> \phi(n)\}} \left|\frac{P\left(V_n = k\right)}{P(V_n=0) \beta_n^{k}/k!} - 1 \right|\longrightarrow 0
 $$

\end{thm}

\begin{proof}

\[
\begin{split}
P & \left(V_n = k\right)   =  \sum_{\substack{B\subset \{1,2,\cdots , n\}\\ |B| = k}} \prod_{i\in B} b(i;n)\, \prod_{j\in B^c} (1-b(j;n))\\
 & =   \sum_{\substack{B\subset \{1,2,\cdots , n\}\\ |B| = k}} \prod_{i\in B} \frac{b(i;n)}{1-b(i;n)}\, \prod_{j=1}^{n} (1-b(j;n))\\
 & =  P(V_{n}=0) \sum_{\substack{B\subset \{1,2,\cdots , n\}\\ |B| = k}} \prod_{i\in B} \frac{b(i;n)}{1-b(i;n)}\\
 & =  P(V_n=0)    \sum_{\substack{A\subset \{1,2,\cdots , n\}\\ |A| = k-1}} \prod_{i\in A} \frac{b(i;n)}{1-b(i;n)}\,  \frac{\sum_{j\in A^c} \frac{b(j;n)}{1-b(j;n)}}{k}\\
 & =  P(V_n=0)    \sum_{\substack{A\subset \{1,2,\cdots , n\}\\ |A| = k-1}} \prod_{i\in A} \frac{b(i;n)}{1-b(i;n)} \frac{\beta_n -\sum_{i \in A} \frac{b(i;n)}{1-b(i;n)}}{k}.
\end{split}
\]

 To obtain a lower bound, we bound the above sums as follows,

\[
\begin{split}
\sum_{\substack{A\subset \{1,2,\cdots , n\}\\ |A| = k-1}} & \prod_{i\in A} \frac{b(i;n)}{1-b(i;n)} \frac{\beta_n -\sum_{i \in A} \frac{b(i;n)}{1-b(i;n)}}{k}\\
& \geq  \> \frac{1}{k}\left(\beta_n - \frac{(k-1) \beta}{(1-\beta)}\right) \sum_{\substack{A\subset \{1,2,\cdots , n\}\\ |A| = k-1}} \prod_{i\in A} \frac{b(i;n)}{1-b(i;n)}\\
&  =   \> \frac{\beta_n}{k}\left(1 - \frac{(k-1)\beta}{\beta_n(1-\beta)}\right) \sum_{\substack{A\subset \{1,2,\cdots , n\}\\ |A| = k-1}} \prod_{i\in A} \frac{b(i;n)}{1-b(i;n)}\\
&  \geq    \> \frac{\beta_n}{k}\left(1 - \frac{(k-1)\beta}{\lambda_n(1-\beta)}\right) \sum_{\substack{A\subset \{1,2,\cdots , n\}\\ |A| = k-1}} \prod_{i\in A} \frac{b(i;n)}{1-b(i;n)}.
\end{split}
\]

 Following similar calculations as those leading up to (\ref{lb2}), we obtain

\[
 \sum_{\substack{B\subset \{1,2,\cdots , n\}\\ |B| = k}} \prod_{i\in B} \frac{b(i;n)}{1-b(i;n)} \geq 
  \frac{\beta_n^k}{k!}\, \left(1- \, \frac{k\, \beta}{\lambda_n\, (1- \beta)}\right)^k .
\]

 Hence, by exploiting  again the increasing property of $(1-\frac{x}{k})^k$ for $0<x<1$ as a function of  $k$,

\[
\begin{split}
P(V_n =k)  \geq & P(V_n=0) \frac{\beta_n^k}{k!}\left(1 - \frac{\frac{k^2\beta}{\lambda_n(1-\beta)}}{k}\right)^k\\
 \geq  & P(V_n=0) \frac{\beta_n^k}{k!}\left(1 - \frac{k^2\beta}{\lambda_n(1-\beta)}\right).
 \end{split}
\]

 To obtain an upper bound,

\[
\begin{split}
P(V_n=0)   & \sum_{\substack{A\subset \{1,2,\cdots , n\}\\ |A| = k-1}} \prod_{i\in A} \frac{b(i;n)}{1-b(i;n)}  \frac{\beta_n -\sum_{i \in A} \frac{b(i;n)}{1-b(i;n)}}{k}\\
& \leq P(V_n=0) \> \frac{\beta_n}{k} \sum_{\substack{A\subset \{1,2,\cdots , n\}\\ |A| = k-1}} \prod_{i\in A} \frac{b(i;n)}{1-b(i;n)}. 
\end{split}
\]

 So,
\[
P(V_n=k) \leq  P(V_n=0) \> \frac{\beta_n^k}{k!}.
\]

 We now have

\begin{eqnarray*} \nonumber 
-\epsilon(k,\lambda_n) \leq \frac{P\left(V_n = k\right)}{P(V_n=0) \beta_n^{k}/k!} -1 \leq 0,
\end{eqnarray*}

 where

\[
\epsilon(k,\lambda_n) = \frac{k^2\beta}{\lambda_n(1-\beta)}.
\]

 For any given $\phi (n)$ such that ${\displaystyle  \phi(n)/\lambda_n \longrightarrow 0}$ as $n\rightarrow +\infty$,

\[
\lim_{n\rightarrow +\infty} \sup_{\{k: \> k^2 \leq \phi(n)\}} \> \epsilon(k,\lambda_n) \to 0.
\]

 This and the validity of our result in the case $k=0$ completes the proof.

\end{proof}

\begin{rem} Since the only assumption on the behavior of $m_n$ is to be bounded away from 1, Theorem \ref{Independent_2}  gives a new local  limit theorem for arbitrary sums of independent Bernoulli random variables with infinite expectation. Note that the range of values of $k$ for which  Theorem \ref{Independent_2} holds differs from the range of values in (\ref{intro-2}) and (\ref{intro-1}). Hence our result provides new limiting results  in the case of sums of Bernoulli random variables.
\end{rem}

We single out an interesting special case of Theorems \ref{Independent_1} and \ref{Independent_2}, where exact Poisson asymptotic behavior is obtained.

\begin{thm} \label{Independent_3}
Let $V_{n}$ be a random variable as defined in (\ref{v}) where the $X_{i,n}$ are independent. If \begin{enumerate}
\item[(A3)]  ${\displaystyle \lim_{n\rightarrow +\infty} \lambda_n = +\infty}$,

\item[(A4)]  ${\displaystyle \lim\limits_{n \to \infty} \sum_{i=1}^{n} b^2(i;n)=0}$,

\end{enumerate}

 then $m_n \to 0$, and for every function $\phi(n)$ such that $\phi(n)m_n \to 0$ as $n  \to +\infty$, 
 $$ \sup_{\{k:\,k^2 \leq \phi(n)\}} \left|\frac{P\left(V_n = k\right)}{e^{-\lambda_n} \lambda_n^{k}/k!} - 1 \right|\longrightarrow 0.
 $$

\end{thm}

\begin{proof}
$(A4)$ and the inequality 
\begin{equation} \nonumber 
\sum_{i=1}^n b^2(i;n) \geq m_n^2
\end{equation}
give that $\lim\limits_{n \to \infty} m_n =0$, and therefore in our proof we can assume that  for all $n$, and for $i$ such that $1 \leq i \leq n$,  $0 \leq b(i;n) < 1/2$. 

 We first consider  the case $k=0$. By using the Taylor series expansion of $\ln (1-x)$, one  obtains the following bounds for $0 < x < 1/2$,

\begin{equation} \label{bounds1}
-x-x^2  \leq \ln (1-x) \leq - x.
\end{equation}

 Using (\ref{bounds1}) with $x = b(i;n)$,  one gets

\begin{eqnarray} \nonumber 
\exp(-\sum\limits_{i=1}^n b^2(i;n)) \leq \frac{P\left(V_{n}=0\right)}{e^{-\, \lambda_n}}  & \leq & \frac{\exp({-\sum\limits_{i=1}^n b(i;n)})}{e^{-\, \lambda_n}} =1.
\end{eqnarray}
So we have that

\begin{eqnarray}\label{panino}
\exp(-\sum\limits_{i=1}^n b^2(i;n)) \leq \frac{P(V_n=0)}{e^{-\lambda_n}} \leq 1.
\end{eqnarray}

 It follows from $(A4)$ that both bounds go to 1 as $n$ goes to infinity, and thus, the convergence holds for the case $k=0$.

 Since $m_n \to 0$, for $k \geq 1$, the proof of Theorem \ref{Independent_1} applies, and (\ref{bounds}) and (\ref{panino}) give

\begin{eqnarray*}
1-
\epsilon_{1}(k,m_n)  \leq \frac{P(V_n=k)}{e^{-\lambda_n}\frac{\lambda_n^k}{k!}} \leq \exp(\sum\limits_{i=1}^n b^2(i;n))\left(1+ \epsilon_{2}(k,m_n)\right) .
\end{eqnarray*}

 The result follows from the above inequalities.

\end{proof}

\subsection{Dependent Case}

In this section we extend our previous results to sums of dependent Bernoulli random variables which satisfy a scheme described by  \cite{Seb}. 

 We use the notation introduced in the Preliminaries if the random variables are independent and modify this notation by adding '\hspace{0.05cm} $\tilde{}$  \hspace{0.05cm}' when the the random variables are dependent. 
Let

\begin{equation*}
S_{k,n} = \sum_{\substack{(i_{1},i_{2},...,i_{k}) \in B\\ B\subset \{1,2,\cdots , n\}\\ |B| = k}} b(i_{1};n)b(i_{2};n)\cdots b(i_{k};n)
\end{equation*}

 and

\begin{equation} \nonumber  
\tilde{S}_{k,n} = \sum_{\substack{(i_{1},i_{2},...,i_{k}) \in B\\ B\subset \{1,2,\cdots , n\}\\ |B| = k}} \tilde{b}(i_{1},i_{2},...,i_{k};n),
\end{equation}

 where the sum is taken over all collections of $k$ distinct indices.

 We consider \textit{rare sets}, $I_k(n)$, to be particular collections of $k$ distinct indices.

\begin{thm} \label{Dependent}
let $\tilde{V}_{n}$ be a random variable defined as in (\ref{v}) and let $V_n$ be a random variable defined as in (\ref{v}) where the $X_{i,n}$ are independent. If the following assumptions are satisfied:

\begin{enumerate}
\item[(B1)] 

\[
\lim\limits_{\substack{n \to +\infty}} \frac{\tilde{b}(i_{1},i_{2},...,i_{k};n)}{b(i_{1};n)b(i_{2};n)\cdots b(i_{k};n)} = 1 \hspace{0.5cm }\textnormal{for all} \hspace{0.2cm} (i_{1},i_{2},...,i_{k}) \in I^{c}_{k}(n)
\]

\item[(B2)] 

\[
\lim_{n \to +\infty} \frac{\tilde{S}_{k,n}}{\sum\limits_{\substack{I_{k}(n)^{c}}} \tilde{b}(i_{1},i_{2},...,i_{k};n)} = 1 
\]

\item[(B3)]

\[
\lim_{n \to +\infty}  \frac{S_{k,n}}{\sum\limits_{\substack{I_{k}(n)^{c}}} b(i_{1};n)b(i_{2};n)\cdots b(i_{k};n)} = 1,
\]

\end{enumerate}

  where the limits (B1), (B2), and (B3) are uniform in $k$ and in $(i_1,i_2,\cdots, i_k) \in I_k^c(n)$ , then

\begin{equation} \label{dep}
\lim_{n \to +\infty} \frac{P(\tilde{V}_{n}=k)}{P(V_n=k)} = 1
\end{equation}

 uniformly in $k$. Moreover Theorems \ref{Independent_1}, \ref{Independent_2}, and \ref{Independent_3} apply with $V_n$ replaced by  $\tilde{V}_n$.

\end{thm}

\begin{rem}
There are some differences between the assumptions of  Theorem \ref{Dependent} and those of  Theorem 1  in \cite{Seb}. In our theorem we require  $(B1)-(B3)$ to hold uniformly in $k$, which is not required in Theorem 1 of \cite{Seb}. Since the proof in \cite{Seb} is based on the method of moments, finiteness is a crucial assumption. By requiring uniformity in $k$, we are able to extend Sevast'yanov's  result to  $\lambda = +\infty$. Moreover, since we do not assume that  $m_n \rightarrow  0$, Theorem \ref{Dependent} also applies to sums of dependent Bernoulli random variables and not only to triangular arrays. In the case $0<\lambda < +\infty$, conditions $(B1)-(B3)$ of Theorem \ref{Dependent} are equivalent to conditions (2) and (3) of Theorem 1  in \cite{Seb}. In the setting where $\lambda$ is not assumed to be finite, our formulation is preferable since when $\lambda = + \infty$, conditions $(B1)-(B3)$ can be satisfied even without assuming
\[
\sum\limits_{\substack{I_{k}(n)}} b(i_{1};n)b(i_{2};n)\cdots b(i_{k};n) \longrightarrow 0
\]
or
\[
\sum\limits_{\substack{I_{k}(n)}} \tilde{b}(i_{1},i_{2},...,i_{k};n) \longrightarrow 0
\]
as \(n\rightarrow +\infty.\)

\end{rem}

\begin{proof}

 We use inclusion-exclusion, see \cite{it}, to express

\begin{equation} \label{prob}
P(\tilde{V}_{n}=k) = \sum\limits^{n}_{\substack{l=0}} (-1)^{l}  
\binom{k+l}{k} \tilde{S}_{k+l,n}.
\end{equation}

 We will compare this to

\begin{equation} \label{ind}
P(V_{n}=k) = \sum\limits^{n}_{\substack{l=0}} (-1)^{l}  
\binom{k+l}{k} S_{k+l,n}.
\end{equation}

 We begin by rewriting
\[
\begin{split}
P&(\tilde{V}_{n}=k)  =  \sum\limits^{n}_{\substack{l=0}} (-1)^{l}  
\binom{k+l}{k} \tilde{S}_{k+l,n} \\
& =\sum\limits^{n}_{\substack{l=0}} (-1)^{l}  
\binom{k+l}{k} \left( \sum\limits_{\substack{I^{c}_{k+l}(n)}} 
\tilde{b}(i_{1},i_{2},...,i_{k+l};n) + \sum\limits_{\substack{I_{k+l}(n)}} 
\tilde{b}(i_{1},i_{2},...,i_{k+l};n) \right)\\
&=  \sum\limits^{n}_{\substack{l=0}} (-1)^{l}  \binom{k+l}{k}
\sum\limits_{\substack{I^{c}_{k+l}(n)}} 
\tilde{b}(i_{1},i_{2},...,i_{k+l};n)\left( 1 + \frac{\sum\limits_{\substack{I_{k+l}(n)}} 
\tilde{b}(i_{1},i_{2},...,i_{k+l};n)}{\sum\limits_{\substack{I^{c}_{k+l}(n)}} 
\tilde{b}(i_{1},i_{2},...,i_{k+l};n)} \right).
\end{split}
\]

 Since we can write

\[
\begin{split}
\sum  \limits_{\substack{I^{c}_{k+l}(n)}} & 
\tilde{b}(i_{1},i_{2},...,i_{k+l};n)\\
& = \sum\limits_{\substack{I^{c}_{k+l}(n)}} 
\left ( \frac{\tilde{b}(i_{1},i_{2},...,i_{k+l};n)}{b(i_{1};n)b(i_{2};n)\cdots b(i_{k+l};n)} \cdot b(i_{1};n)b(i_{2};n)\cdots b(i_{k+l};n)\right ),
\end{split}
\]

 it follows from $(B1)$ that for $\epsilon >0$ there exists $n_{0}$ such that for $n \geq n_{0}$ and all $l+k$, 
\[
\begin{split}
 (1-\frac{\epsilon}{3})  \sum\limits_{\substack{I^{c}_{k+l}(n)}} & b(i_{1};n)b(i_{2};n)\cdots b(i_{k+l};n)\\
 & \leq   \sum\limits_{\substack{I^{c}_{k+l}(n)}} 
\tilde{b}(i_{1},i_{2},...,i_{k+l};n)\\
& \leq   (1 + \frac{\epsilon}{3}) \sum\limits_{\substack{I^{c}_{k+l}(n)}} b(i_{1};n)b(i_{2};n)\cdots b(i_{k+l};n).
\end{split}
\]

 From  $(B3)$, it follows that there exists an $n_{1}$ such that for $n\geq n_{1}$ and all $l+k$,

\[
 (1-\frac{\epsilon}{3}) \leq \frac{\sum\limits_{\substack{I^{c}_{k+l}(n)}} b(i_{1};n)b(i_{2};n)\cdots b(i_{k+l};n)}{S_{k+l,n}} \leq (1 + \frac{\epsilon}{3}).
\]

 Lastly, since

\[
\frac{\tilde{S}_{k,n}}{\sum\limits_{\substack{I^{c}_{k+l}(n)}} 
\tilde{b}(i_{1},i_{2},...,i_{k+l};n)}=1+ \frac{\sum\limits_{\substack{I_{k+l}(n)}} 
\tilde{b}(i_{1},i_{2},...,i_{k+l};n)}{\sum\limits_{\substack{I^{c}_{k+l}(n)}} 
\tilde{b}(i_{1},i_{2},...,i_{k+l};n)},
\]

 it follows from $(B2)$ that there exists $n_{2}$ such that for $n \geq n_{2}$ and all $l+k$,

\[
1 \leq 1+ \frac{\sum\limits_{\substack{I_{k+l}(n)}} 
\tilde{b}(i_{1},i_{2},...,i_{k+l};n)}{\sum\limits_{\substack{I^{c}_{k+l}(n)}} 
\tilde{b}(i_{1},i_{2},...,i_{k+l};n)} \leq (1 + \frac{\epsilon}{3}).
\]

 Let $N = \max\{n_{0}, n_{1}, n_{2} \}$. Then for $n \geq N$,
\[
\begin{split}
(1-\frac{\epsilon}{3})^{2} & \sum\limits^{n}_{\substack{l=0}} (-1)^{l}  
\binom{k+l}{k} S_{k+l,n} \leq P(\tilde{V}_{n}=k)\\
&  \leq (1 + \frac{\epsilon}{3})^{3} \sum\limits^{n}_{\substack{l=0}} (-1)^{l}  
\binom{k+l}{k} S_{k+l,n}
\end{split}
\]

 which can be expressed as 

\[
(1-\frac{\epsilon}{3})^{2} P(V_{n}=k) \leq P(\tilde{V}_{n}=k) \leq  (1 + \frac{\epsilon}{3})^{3} P(V_{n}=k). 
\]

 Thus, for all $k$,

\[
(1-\frac{\epsilon}{3})^{2} \leq \frac{P(\tilde{V}_{n}=k)}{P(V_{n}=k) } \leq (1 + \frac{\epsilon}{3})^{3}. 
\]

 Since $\epsilon$ is arbitrary, we have shown that the results in Theorems \ref{Independent_1}, \ref{Independent_2}, and \ref{Independent_3} hold under the respective assumptions for this case of dependence.

\end{proof}

\section{Acknowledgements}

 The first author would like to thank Professor Janos Galambos for suggesting the study of Poisson-type local limit theorems for $V_n$ in the case of infinite expectation.

\end{document}